\documentclass[a4paper,reqno,11pt]{amsart}

\usepackage[indent=2em, skip=0.5em]{parskip}

\usepackage[utf8]{inputenc}
\usepackage{microtype}
\usepackage{amsmath,amsthm,amsfonts,amssymb}
\usepackage{color}
\usepackage{graphicx}
\usepackage[margin=0.9in]{geometry}

\usepackage[usenames,dvipsnames]{xcolor}
\usepackage{hyperref}
\hypersetup{
    colorlinks=true,
    linkcolor=cyan!80!black,
    citecolor=MidnightBlue,
    urlcolor=magenta,
}
\usepackage{tikz}
\usetikzlibrary{fit}
\usepackage{booktabs}

\newcommand\addvmargin[1]{
  \node[fit=(current bounding box),inner ysep=#1,inner xsep=0]{};
}

\theoremstyle{plain}
\newtheorem{theorem}{Theorem}[section]

\newtheorem{lemma}{Lemma}[section]

\theoremstyle{definition}

\newtheorem{fact}{Fact}[]
\newtheorem{assumption}{Assumption}[]
\newtheorem{example}{Example}[section]

\theoremstyle{remark}
\newtheorem{remark}{Remark}[section]

\newcommand{\E}{\mathbb{E}}

\renewcommand{\Pr}{\mathbb{P}}
\newcommand{\C}{\mathcal{C}}

\newcommand{\Z}{\mathbb{Z}}

\newcommand{\tr}{\mathrm{tr}}

\newcommand{\X}{\mathbb{X}}
\newcommand{\Tr}{\mathrm{Tr}}

\newcommand{\appeq}{\overset{.}{=}}
\newcommand{\mathemdash}{\text{---}}

\let\tilde\widetilde
\let\hat\widehat

\numberwithin{equation}{section}

\title{The ``visible" Wigner matrix}

\author{Arup Bose}
\author{Soumendu Sundar Mukherjee}
\address{
    Statistics and Mathematics Unit \\
    Indian Statistical Institute \\
    203 B. T. Road, Kolkata 700108
}
\email{bosearu@gmail.com}
\email{ssmukherjee@isical.ac.in}
\date{}

\begin{document}

\begin{abstract} 
We consider the ``visible'' Wigner matrix, a Wigner matrix whose $(i, j)$-th entry is coerced to zero if $i, j$ are co-prime. Using a recent result from elementary number theory on co-primality patterns in integers, we show that the limiting spectral distribution of this matrix exists, and give explicit descriptions of its moments in terms of infinite products over primes $p$ of certain polynomials evaluated at $1/p$. We also consider the complementary ``invisible'' Wigner matrix.
\end{abstract}
\maketitle

\section{Introduction}
Suppose $A_n$ is any $n\times n$ real symmetric matrix with random entries and whose eigenvalues (written in some order) are $\lambda_1, \lambda_2, \ldots, \lambda_n$. The random probability law which puts mass $n^{-1}$ on each eigenvalue $\lambda_i$ is called the \textit{empirical spectral measure} of $A_n$ and the corresponding distribution function $F_{A_n}$ is called the \textit{empirical spectral distribution function} (ESD). The {\it expected empirical spectral distribution function} (EESD) is defined as
\[
    \E[F_{A_n}(x)]=\frac{1}{n}\sum_{i=1}^{n}\mathbb{P}(\lambda_i  \le x).
\]
We shall write $\E(F_{A_n})$ in short. The corresponding probability law is called the {\it expected empirical spectral measure}.

Suppose the entries of $\{A_n\}$ are defined on some probability space $(\Omega, \mathcal{F}, \Pr)$. Let $F$ be a \textit{non-random} distribution function on $\mathbb{R}$. Let $C_F$ denote the set of continuity points of $F$. We say that the ESD of $A_n$ converges to $F$ \textit{weakly almost surely} if for almost every $\omega \in \Omega$ and for all  $t\in C_F$,
\[
    F_{A_n}(t)\rightarrow F(t) \ \ \mbox{as} \ n\rightarrow \infty.
\]

It is easy to see that the above convergence  implies that $\E(F_{A_{n}})$ converges weakly to $F$. This $F$ is called the \textit{limiting spectral distribution (or measure)} (LSD) of $\{A_n\}$.

The Wigner matrix is one of the most important matrices studied in Random Matrix Theory. An $n\times n$ real Wigner matrix $ W_n$ is a symmetric matrix with independent and identically distributed (i.i.d.) random entries $W_{ij}$ with mean $0$ and variance $1$. For the moment let us also assume that they have all moments finite. It is well-known that the almost sure LSD of $n^{-1/2}W_n$ is the standard semi-circle law with density
\begin{eqnarray} \label{eqn: semipdf4.1}
f(x) = \begin{cases} \frac{1}{2\pi} \sqrt{4-x^2}\ \ \text{\rm if} \ -2 < x <2, \\
0\ \ \text{\rm otherwise.}
\end{cases}
\end{eqnarray}
The odd moments of the semi-circle law are all $0$, and the $2k$-th moment is the $k$-th Catalan number
\[
    C_k = \frac{1}{k+1} \binom{2k}{k}, \ k \geq 1.
\]
In this article, we introduce a number-theoretic variant $\widetilde W_n$ of $W_n$, where
\begin{equation}\label{eq:vis-wig}
    \widetilde{W}_{ij} = W_{ij} \mathbf{1}_{\gcd(i, j) = 1}, \ 1 \le i, j \le n,
\end{equation}
and call it the \textit{visible Wigner matrix}. The name derives from the theory of visibility of lattice points. A point $(i, j)$ in the lattice $\Z^2$ is called \emph{visible} from the origin $(0, 0)$ if there are no other lattice points on the line segment joining $(0, 0)$ and $(i, j)$. It is a well-known fact from basic number theory that $(i, j)$ is visible if and only if $\gcd(i, j) = 1$. Note that there are still of the order of $n^{2}$ entries in $\widetilde{W}_n$ which are not identically $0$, and hence the matrix cannot be considered as sparse.

In a similar vein, the corresponding \textit{invisible Wigner matrix} $\widehat{W}_n$ is defined as follows:
\begin{equation}\label{eq:invis-wig}
    \widehat{W}_n=((\widehat{W}_{ij}))_{1\leq i, j \leq n} = ((W_{ij} \mathbf{1}_{\gcd(i, j) \neq 1}))_{1\leq i, j \leq n}\,.
\end{equation}
We show that the ESDs of both $n^{-1/2}\widetilde W_n$ and $n^{-1/2}\widehat{W}_n$ converge weakly almost surely. A comparison of the histograms of the ESD of a visible Wigner,  an invisible Wigner, and a standard Wigner matrix is provided in Figure~\ref{fig:esd}.

\begin{figure}[ht]
    \centering
    \includegraphics[scale = 0.4]{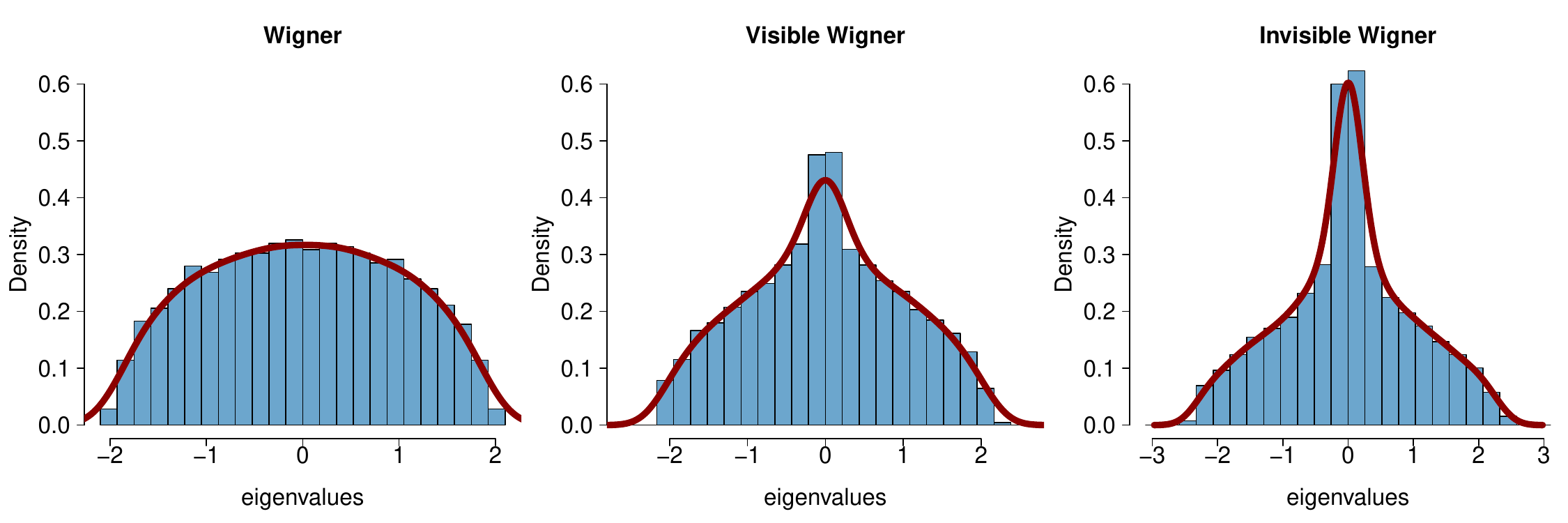}
    \caption{Histogram and kernel density estimate of the ESDs of $n^{-1/2}W_n$, $n^{-1/2}\tilde{W}_n$ and $n^{-1/2}\hat{W}_n$, for $n = 1000$ and with standard Gaussian entries. The eigenvalues of $n^{-1/2}\tilde{W}_n$ and $n^{-1/2}\hat{W}_n$ have been normalised by their limiting standard deviations, $\sqrt{6/\pi^2}$ and $\sqrt{1 - 6/\pi^2}$, respectively, so that all three LSDs have unit second moment.}
    \label{fig:esd}
\end{figure}

The second moment of the LSD of $n^{-1/2}\tilde{W}_n$ can be computed without much effort. Letting $\phi(\cdot)$ denote \textit{Euler's totient function}, we have
\begin{align*}
    \frac{1}{n}\E\big[\Tr(n^{-1/2}\tilde{W}_n)^2\big]
        &= \frac{1}{n^2}\E\sum_{1\leq i, j\leq n} \widetilde{W}_{n, ij}^2 \\
		&= \frac{1}{n^2}\sum_{\substack{1\leq i, j\leq n \\ \gcd(i, j) = 1}} 1 \\
        &= \frac{1}{n^2} \bigg(2 \sum_{1\leq j \le n} \phi(j) - 1\bigg) \\
        &= \frac{6}{\pi^2} + O\bigg(\frac{\log n}{n}\bigg),
\end{align*}
where we have used the following elementary fact from number theory that (see, for example, page 36 of \cite{montgomery2007}):
\[
    \sum_{1\leq j \le n} \phi(j) = \frac{3n^2}{\pi^2} + O(n\log n).
\]
Therefore the second moment of the LSD of $n^{-1/2}\tilde{W}_n$ is $\frac{6}{\pi^2}$, and a fortiori, the second moment of the LSD of $n^{-1/2}\widehat W_n$ is $1 - \frac{6}{\pi^2}$.

In this context, note that if $X_1, X_2$ are i.i.d. discrete uniform variables over $[n] := \{1, \ldots, n\}$, then
\begin{equation}\label{eq:gcdheuristic}
    \Pr(\gcd(X_1, X_2) = 1) = \frac{1}{n^2}\sum_{\substack{1\leq i, j\leq n \\ \gcd(i, j) = 1}} 1 = \frac{1}{n}\E\big[\Tr(n^{-1/2}\tilde{W}_n)^2\big] \to \frac{6}{\pi^2}.
\end{equation}
A heuristic probabilistic explanation of (\ref{eq:gcdheuristic}) can be given via Cram\'er's model in number theory, which posits that divisibility by different primes are independent events. Thus, denoting \emph{heuristic equality} by $\appeq$, we have
\begin{align*}
    \lim_{n \to \infty} \Pr(\gcd(X_1, X_2) = 1)
    &= \lim_{n \to \infty} \Pr(\text{no prime } p \le n \text{ divides } \gcd(X_1, X_2)) \\
    &\appeq \lim_{n \to \infty} \prod_{p\ \text{prime},\ p \le n} \Pr(p \text{ does not divide } \gcd(X_1, X_2)) \\
    &= \lim_{n \to \infty} \prod_{p\ \text{prime},\ p \le n} \big(1 - \Pr(p \text{ divides } \gcd(X_1, X_2)\big) \\
    &= \lim_{n \to \infty} \prod_{p\ \text{prime},\ p \le n} \big(1 - \Pr(p \text{ divides } X_1) \Pr(p \text{ divides } X_2)\big) \\
    &= \lim_{n \to \infty} \prod_{p\ \text{prime},\ p \le n} \bigg(1 - \frac{1}{p^2} + O\bigg(\frac{1}{n}\bigg)\bigg) \\
    &= \prod_{p\ \text{prime}} \bigg(1 - \frac{1}{p^2}\bigg) \\
    &= \frac{1}{\zeta(2)}=\frac{6}{\pi^2},
\end{align*}
where the penultimate equality follows from Euler's product formula for the Riemann zeta function $\zeta(\cdot)$.

In general, the limit moments can be described in terms of \textit{independence polynomials} of certain trees. Consider a finite graph $G = (V, E)$, where $V$ is the set of vertices and $E$ is the set of edges. Recall that an \textit{independent set} in a graph is a collection of vertices between which no edges exist. Let
\[
    i_m(G) = \text{number of independent sets of }  G \text{ of size } m.
\]
Then the independence polynomial of $G$ is defined as
\[
    I_G(z) = \sum_{m = 0}^{\alpha(G)}i_m(G) z^m,
\]
where $\alpha(G)$ denotes the \textit{independence number} (the size of the largest independent set) of $G$.

Determining the probability that randomly chosen integers are co-prime is a well-studied problem in elementary number theory \cite{nymann1972probability,laszlo2002probability,hu2014pairwise,de2015counting}. Imagining that the vertices of a finite graph carries independent uniformly random integers in $[n]$, \cite{hu2014pairwise} and \cite{de2015counting} consider the question of determining the limiting probability, as $n \to \infty$, that these random integers are pairwise co-prime along the edges of the graph, i.e. the two random integers at the end-vertices of any edge are co-prime. The following polynomial appears in these probability expressions.
With $|V|$ denoting the number of elements of $V$, let
\begin{align}\label{eq:poly_def}
    Q_G(z) &:= \sum_{m = 0}^{|V|} i_m(G) (1 - z)^{|V| - m} z^m \\ \nonumber
    &= (1 - z)^{|V|} I_G((1 - z)^{-1} z)\,.
\end{align}
Now let $X_1, \ldots, X_k$ be i.i.d.~and uniformly distributed on $[n]$. Let $G = (V, E)$ be a graph on $k$ vertices labeled $1, \ldots, k$. Then \cite{hu2014pairwise} shows that as $n \to \infty$,
\begin{equation}\label{eq:hu2014}
    \Pr\big(\gcd(X_i, X_j) = 1 \text{ for all } 1\leq i, j \leq k, \ (i, j) \in E\big)= A_{G} + O\bigg(\frac{(\log n)^{k - 1}}{n}\bigg),
\end{equation}
where
\[
    A_{G} = \prod_{p \,\, \text{prime}}Q_{G}\bigg(\frac{1}{p}\bigg).
\]
A different expression for $A_G$ was obtained in \cite{de2015counting}. To give an example, if $G$ is the complete graph on $k$ vertices, then it is easy to see that
\[
    Q_G(z) = (1 - z)^k + k (1 - z)^{k - 1} z = (1 - z)^{k - 1}(1 + (k - 1)z).
\]
Thus
\[
    A_G = \prod_{p \,\, \text{prime}}\bigg(1 - \frac{1}{p}\bigg)^{k - 1} \bigg(1 + \frac{k - 1}{p}\bigg).
\]
Here $A_G$ is the limiting probability that $k$ randomly chosen integers are pairwise co-prime. The above expression for $A_G$ was obtained in an earlier work of L. T\'oth \cite{laszlo2002probability}.

We will see that the moments of the LSD of $n^{-1/2}\widetilde{W}_n$ are sums of different $A_G$'s where $G$ runs through a certain family of trees. The moments of the LSD of $n^{-1/2}\widehat{W}_n$ have more intricate representations in terms of various $A_G$'s, where $G$ runs through a certain family of forests (recall that a forest is a disjoint union of trees).

\section{Preliminaries}
First, we state a convenient assumption on the moments of the entries of our random matrices.

\begin{assumption} \label{assmp:entries}
For every $n$, $\{x_{ij}=x_{ij,n}, i\leq j \}$ are independent with mean zero and variance 1. Moreover,
\[
    \sup_{i, j, n} \E(|x_{i,j,n}|^k) \leq B_k < \infty, \ \ \text{for all}\ \ k\geq 1.
\]
\end{assumption}

The $h$-th moment of  $F_{A_n}$ is given by
\begin{equation}\label{eq:tracemoment}
    m_h(F_{A_n}) = \frac{1}{n} \sum_{i=1}^n \lambda_i^h = \frac{1}{n} \mathrm{Tr}(A_n^h) =\tr (A_n^h)=  m_h (A_n) \mbox{ (say)},
\end{equation}
where $\mathrm{Tr}$ and $\mathrm{tr}$ denote the trace and the normalised trace. Relation \eqref{eq:tracemoment} is known as the \textit{trace-moment} formula. Further, the moments of the EESD are given by
\[
    m_h(\E(F_{A_{n}}))=\E\big[\frac{1}{n} \mathrm{Tr}(A_n^h)\big]=\E[m_h(A_n)].
\]
We now give some basic results pertaining to the \emph{moment method} which is what we will use to establish the existence of the LSD. Consider the following conditions:\vskip5pt

\noindent (M1) For every $h \geq 1$,  $\E[m_h (A_n)] \rightarrow
m_h$, which is finite. \vskip3pt

\noindent (U) The moment sequence  $\{m_h\}$ corresponds to a unique probability law.\vskip3pt

\noindent (M4) For every $h \geq 1$,  $\displaystyle{\sum_{h=1}^\infty \E\big[m_h (A_n)-\E[m_h (A_n)]\big]^4 < \infty}$.
\vskip5pt

Incidentally, the sequence $\{m_h\}$ is automatically a moment sequence of some probability distribution. Condition (U) insists that this distribution be unique.

It is easy to see that if (M1) and (U) hold, then the EESD of $A_n$ converges in distribution to the distribution function $F$ determined by its moments $\{m_h\}$. If further (M4) holds, then the ESD converges to $F$ almost surely. See Lemma 1.2.4 of \cite{bose2018} for a proof.

Computation of $m_h$ involves identifying the ``leading'' terms and calculating their contribution in the following expansion ($x_{i,j}$ is the $(i,j)$-th entry of $A_n$):
\[
    \E[\mbox{tr}(A_n^h)]=\sum_{1\leq i_1,i_2,\ldots,i_h\leq n}\E[x_{i_1,i_2}x_{i_2,i_3}\cdots x_{i_{h-1},i_h}x_{i_h,i_1}].
\]
See \cite{bose2018} for details of this  approach and its application to several random matrices. We need the following developments from there for our purposes.

For fixed $h$ and $n$, a \textit{circuit} of \textit{length} $h$ is a function \  $\pi:\{0,1,2,\ldots,h\} \rightarrow \{1,2,\ldots,n\} $ \ with  $\pi(0) = \pi(h)$. Then (M1) can be written as
\begin{equation}\label{eq:m1trace}
    \E[m_{h}(n^{-1/2} A_n)] = \E\bigg[\tr\bigg(\frac{A_n}{\sqrt n}\bigg)^h\bigg] = \frac{1}{n^{1+h/2}} \sum_{\pi: \ \pi \ \text{circuit}}\hspace{-5pt} \E \X_{\pi}\rightarrow m_h,
\end{equation}
where
\[
    \X_{\pi} = x_{\pi(0), \pi(1)} x_{\pi(1),\pi(2)} \cdots x_{\pi(h-2),
\pi(h-1)} x_{\pi(h-1), \pi(h)}.
\]
A circuit $\pi$ is \textit{matched} if each \textit{unordered} pair $(\pi(i-1),\pi(i))$, $1\leq i \leq h$, is repeated at least twice. If $\pi$ is non-matched, then $\E(\X_{\pi})=0$ and does not contribute to (\ref{eq:m1trace}). If each value is repeated \textit{exactly twice} (so $h$ is even), then $\pi$ is \textit{pair-matched}\index{circuit, pair-matched} and $\E(\X_{\pi})=1$. 

\begin{fact}\label{fact:pm_only}
    Under Assumption~\ref{assmp:entries}, by an easy counting argument, it can be shown that the entire contribution from non-pair-matched circuits is negligible for $n^{-1/2} W_n$ (see \cite{bose2018} for details). By the same argument, this is also the case for $n^{-1/2} \widehat{W}_n$ and $n^{-1/2} \widetilde{W}_n$. Hence we need to consider only the cases where $h$ is even. In addition, the LSD of $n^{-1/2} W_n$ is determined by the \textit{counts} of different types of pair-matched circuits.
\end{fact}

For any fixed length $h$, we group the circuits into \textit{equivalence classes} via the following equivalence  relation: $\pi_1\sim\pi_2$\index{circuit, equivalence} \index{equivalent circuits} if and only if for all unordered pairs
$(i,j)$, $1\leq i, j \leq h$,
\begin{equation}\label{eq: pidef} (\pi_1(i-1), \pi_1(i)) = (\pi_1(j-1),
\pi_1(j)) \Longleftrightarrow (\pi_2(i-1), \pi_2(i)) = (\pi_2(j-1),
\pi(j)).\end{equation}
    \noindent Any equivalence class can be indexed by a partition of $\{1,2,\ldots,h\}$. Each partition block identifies the positions of the matches. We label these partitions by \textit{words} $w$ of
length $h$ of letters where  the first occurrence of each
letter is in alphabetical order. \index{alphabetical order}

As an example, suppose $h = 4$. Consider the partition $\{ \{1,4\}, \{2,3\}\}$. It is then represented by the word $abba$. This identifies all circuits $\pi$ for which $(\pi(0), \pi(1))= (\pi(3), \pi(4))$ and $(\pi(1), \pi(2))=(\pi(2), \pi(3))$ (and it is understood that the two pairs are not equal). This is the same as saying $\pi(0)=\pi(4)$ and $\pi(1)=\pi(3)$. Thus $\pi(0), \pi(1)$ and $\pi(2)$ can be chosen in $n (n-1)(n-2)$ ways. We shall assume that they can be chosen ``freely'', that is in $n^3$ ways, because the additional number of terms become negligible asymptotically after scaling.

Let $w[i]$ denote the $i$-th entry (letter) of $w$. The equivalence class $\Pi_A(w)$ and a related class which we shall need,  corresponding to $w$ are defined by
\begin{align*}
    \Pi_A(w) &= \{\pi: w[i] = w[j] \Leftrightarrow (\pi(i-1),\pi(i)) = (\pi(j-1),\pi(j))\}; \\
    \Pi_A^{\prime}(w) &= \{\pi: w[i] = w[j] \Rightarrow (\pi(i-1),\pi(i)) = (\pi(j-1),\pi(j))\}.
\end{align*}
Notions connected to circuits extend to words in an obvious way. For instance, $w$ is {\it pair-matched} if every letter appears exactly twice. By Fact~\ref{fact:pm_only}, only pair-matched words of even length are relevant.

Any $i$ (or $\pi(i)$ by abuse of notation) will be called a \textit{vertex}. It is \textit{generating}, if either $i=0$ or $w[i]$ is the \emph{first} occurrence of  a letter. Otherwise it is {\it non-generating.}  For example, for $w=abba$, the vertices $\pi(0), \pi(1)$ and $\pi(2)$ are generating, while for $w = ababcc$, $\pi(0),\ \pi(1),\ \pi(2)$ and $\pi(5)$ are generating. Clearly, for any word, it is only the generating vertices that can be freely chosen.

We shall use the following generic notation: For any sequence of matrices $\{A_n\}$, and any word of length $w$,
\[
    p_A(w)=\lim \frac{1}{n^{1+k}} \#\Pi_A(w)=\lim \frac{1}{n^{1+k}} \#\Pi_A^{\prime}(w),
\]
whenever the limits exist and are equal.

Note that for the Wigner matrix,
\begin{equation}\label{eq:c1c2}
    w[i]=w[j] \iff
        \begin{cases}
            \text{either}\ (\pi(i-1)=\pi(j),   \pi(i)=\pi(j-1)) \ \text{(called a (C2) constraint)}&\\
            \text{or}\ (\pi(i-1)=\pi(j-1),  \pi(i)=\pi(j))\ \text {(called a (C1) constraint)}.&\\
        \end{cases}
\end{equation}
A pair-matched $w$ is called \textit{Catalan} if there is at least one double letter, and successive removal of  double letters leads to the empty word.  For example,  $aabb, abccba$ are Catalan words but $ababcc$ is not. Let $\mathcal{C}(2k)$ denote the set of all Catalan  words of length $2k$. It is well-known that this sub-class of words are the relevant words for the Wigner matrix.\vskip5pt

\begin{fact}\label{fact:pw}
    Under Assumption~\ref{assmp:entries}, for the Wigner matrix $n^{-1/2}W_n$ (see \cite{bose2018} for details),
\begin{equation}\label{eq:pistarwigner}
p_W(w)=\lim \frac{1}{n^{1+k}} \#\Pi_W(w)=\lim \frac{1}{n^{1+k}} \#\Pi_W^{\prime}(w)=
\begin{cases}{0}&\text {if $w\notin \mathcal{C}(2k)$},\\
1&\text {if $w\in  \mathcal{C}(2k)$.}\\
\end{cases}
\end{equation}
Moreover, for every Catalan word, only the (C2) constrained circuits contribute. Further,
\ $\# \mathcal{C}(2k)=C_k$.
This shows that condition (M1) is satisfied with $2k$-th limit moment
\[
    m^{\mathrm{(W)}}_{2k}=\sum_{w\in \mathcal{C}(2k)}p_W(w)
=\frac{1}{k+1} \binom{2k}{k},
\]
and the odd moments $m^{\mathrm{(W)}}_{2k+1}=0$ for all $k \geq 0$. These numbers satisfy condition (U) and  identify the limit of the EESD of $n^{-1/2}W_n$ as the semi-circle law. See \cite{bose2018} for details.
\end{fact}

\begin{remark}
    \label{remark:wignerequiv} From the standard arguments to prove  (\ref{eq:pistarwigner}) (see \cite{bose2018}), it immediately follows that for $n^{-1/2}\widehat{W}_n$ and $n^{-1/2}\widetilde W_n$, as in Fact~\ref{fact:pw},  the counts $\#\Pi_{\widehat{W}}(w)$ and $\#\Pi^{\prime}_{\widehat{W}}(w)$, as well as $\#\Pi_{\widetilde{W}}(w)$ and $\#\Pi_{\tilde{W}}^\prime(w)$ continue to remain equivalent, and $p_{\widehat{W}}(w)=0$ and $p_{\widetilde{W}}(w)=0$ if $w\notin \mathcal{C}(2k)$. We shall show later that  $p_{\widehat{W}}(w)$ and $p_{\widetilde{W}}(w)$ exist for each $w\in \mathcal{C}(2k)$ (with zero contribution from the (C1) constrained circuits). That would immediately confirm the convergence of their EESDs.
\end{remark}

\begin{fact} \label{fact:m4}
    Under Assumption~\ref{assmp:entries}, by another counting argument, condition (M4) is valid for $n^{-1/2}W_n$ (see \cite{bose2018} for details). Hence the almost sure limit of the ESD of $n^{-1/2}W_n$ is also the semi-circle law.
\end{fact}

As before, this counting argument, and hence  condition (M4), remains valid for both $n^{-1/2}\widetilde W_n$ and $n^{-1/2}\widehat{W}_n$. As a consequence, the almost sure convergence of their ESDs would follow once we show the convergence of their EESDs.

\noindent \textbf{Catalan words, trees and associated polynomials}.
In order to describe the limit moments of the visible and the invisible Wigner matrices, we need to tie up Catalan words with certain trees and the associated polynomials $Q$ introduced earlier. Suppose $w$ is a Catalan word of length $2k$. Then there are exactly $(k + 1)$ generating vertices. Note that it is enough to consider only circuits all whose constraints are (C2). Consider the cycle graph $0 \, \mathemdash \, 1 \, \mathemdash \, \cdots \, \mathemdash \, (2k - 1) \, \mathemdash \, 0$. Identify any two vertices $u, v$ if the Catalan word $w$ forces the equality $\pi(u) = \pi(v)$ (via a (C2) constraint), for circuits in $\Pi_W(w)$. This results in a tree $G(w)$ on $(k + 1)$ vertices (see Figure~\ref{fig:gwpic} for an example). These trees also appear when one associates words with Dyck-paths (see, e.g., \cite{andguizei}).

 \begin{figure}[ht]
    \centering
    \includegraphics[scale = 1]{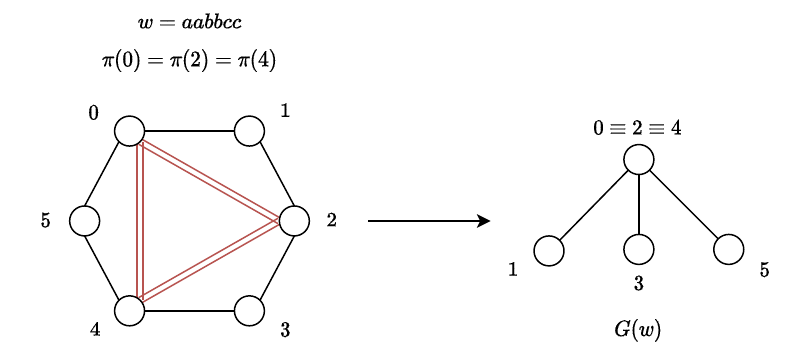}
    \caption{A pictorial description of how $G(w)$ is obtained for $w = aabbcc$ via vertex identification.}
    \label{fig:gwpic}
\end{figure}

For example, suppose $k = 1$. Then there is only one Catalan word $w = aa$. The corresponding tree $G(w)$ is the $2$-path $0 \,\, \mathemdash \,\, 1$, the corresponding unlabeled tree being $\circ \mathemdash \circ$. This tree has one independent set of size $0$ and two independent sets of size $1$. The associated polynomial, defined via (\ref{eq:poly_def}), is
\[
    Q_{G(w)}(z) = (1 - z)^2 + 2 (1 - z) z = 1 - z^2
\]
and
\begin{equation}\label{eq:zeta2}
    A_{G(w)} = \prod_{p\ \text{prime}}\bigg(1-\frac{1}{p^2}\bigg) = \frac{1}{\zeta(2)} = \frac{6}{\pi^2}.
\end{equation}
For $k=2$ there are two Catalan words, $w_1=abba$ and $w_2=aabb$, each with three generating vertices. The associated trees with three vertices are respectively $G(w_1)=0 \,\, \mathemdash \,\, 1 \,\, \mathemdash \,\, 2$ and  $G(w_2) = \{0 \,\, \mathemdash \,\, 1, 0 \,\, \mathemdash \,\, 2\}$. The unlabeled version of both of these trees is the $3$-path $\circ \mathemdash \circ \mathemdash \circ$.
% \begin{center}
% \begin{tikzpicture}
%     \tikzstyle{hollow node}=[circle,draw,inner sep=1.5]
%     \node[hollow node]{}
%         child{node[hollow node]{}}
%         child{node[hollow node]{}}
%     ;
% \end{tikzpicture} $\cdot$
% \end{center}
%$G(w_1)=0\\sim1\sim 2$ and  $G(w_2)=\{0\sim 1, 0\sim 2\}$ which are isomorphic.
Note that for this tree we have one independent set of size $0$, three independent sets of size $1$, and one independent set of size $2$. Therefore
\begin{align*}
    Q_{G(w_1)}(z) &= Q_{G(w_2)}(z) \\
                  &= (1 - z)^3 + 3 (1 - z)^2 z + (1 - z) z^2 \\
                  &= (1 - z) (1 + z - z^2).
\end{align*}
For $k = 3$, we have listed the $5$ Catalan words and the corresponding unlabeled trees in Table~\ref{tab:k3}.
\begin{table}[ht]
    \centering
    \caption{The 5 Catalan words of length $6$ and the associated unlabeled trees.}
    \label{tab:k3}
    \begin{tabular}{c|c|c}
        \toprule
        Word ($w$) & Unlabeled $G(w)$ & $Q_{G(w)}(z)$ \\
        \midrule
        \begin{tabular}{c}
             $aabbcc$ \\
             $abbcca$
        \end{tabular} &
        \begin{tikzpicture}[baseline=-25]
            \centering
            \tikzstyle{hollow node}=[circle,draw,inner sep=1.5]
            \node[hollow node]{}
                child{node[hollow node]{}}
                child{node[hollow node]{}}
                child{node[hollow node]{}}
            ;
            \addvmargin{3mm}
        \end{tikzpicture} &
        $(1 - z)^4 + 4 (1 - z)^3 z + 3 (1 - z)^2 z^2 + (1 - z) z^3$ \\
        \midrule
        \begin{tabular}{c}
             $abccba$ \\
             $abbacc$ \\
             $aabccb$
        \end{tabular} &
        \begin{tabular}{c}
        \begin{tikzpicture}
            \centering
            \tikzstyle{hollow node}=[circle,draw,inner sep=1.5]
            \node[hollow node]{}
                child{node[hollow node]{}}
                child{node[hollow node]{}
                    child{node[hollow node]{}}
                };
            ;
            \addvmargin{3mm}
        \end{tikzpicture}
        \end{tabular} &
        $(1 - z)^4 + 4 (1 - z)^3 z + 3 (1 - z)^2 z^2$ \\
        \bottomrule
    \end{tabular}
\end{table}

For a general $k$, we obtain via this procedure all unlabeled trees on $k + 1$ vertices. Each of these trees appear a specific number of times, corresponding to cyclic traversals of the tree that visit each edge exactly twice. One can thus partition the set of all Catalan words of length $2k$ into $T_k$ many equivalence classes, where $T_k$ is the number of all unlabeled trees on $k + 1$ vertices (the sequence $(T_k)_{k \ge 1}$ appears in the Online Encyclopedia of Integer Sequences (OEIS) as sequence number A000055  \cite{oeis}).

\section{The LSDs of the visible and the invisible Wigner}
The following lemma will be the main workhorse in the proofs of our results given below. For any Catalan word $w$ of length $2k$ and $i \in \{0, \ldots, 2k - 1\}$, let
\[
    B_{i}(w) := \{\pi \in \Pi_W^{\prime}(w) \mid \gcd(\pi(i), \pi(i + 1)) = 1\}.
\]
Note that $B_{i}(w)$ has been defined as a subset of $\Pi_W^{\prime}(w)$ by imposing an additional gcd constraint.
\begin{lemma}\label{lem:limit-subgraph-Gw}
For any $\ell \in \{1, \ldots, 2k\}$ and  any $0 \le i_1 < \cdots < i_{\ell} \le 2k - 1$, the limit
\[
    \lim_{n \to \infty}\frac{\# B_{i_1}(w) \cap \cdots \cap B_{i_{\ell}}(w)}{n^{1 + k}}
\]
exists and equals $A_{H_{w; i_1, \ldots, i_{\ell}}}$, where $H_{w; i_1, \ldots, i_{\ell}}$ is an explicit sub-graph of $G(w)$.
\end{lemma}
\begin{proof}
Any $\pi \in B_{i_1}(w) \cap \cdots \cap B_{i_{\ell}}(w)$ satisfies
\begin{equation}\label{eq:gcd-cons}
    \gcd(\pi(i_j), \pi(i_j + 1)) = 1 \text{ for } j = 1, \ldots, \ell.
\end{equation}
Recall that the tree $G(w)$ was obtained from the cyclic graph $0 \, \mathemdash \, 1 \, \mathemdash \, \cdots \, \mathemdash \, (2k - 1) \, \mathemdash \, 0$ by identifying vertices $u, v$ whenever $\pi(u) = \pi(v)$. The constraints \eqref{eq:gcd-cons} put gcd constraints on a subset $E$ of the edges of $G(w)$. We denote the sub-graph of $G(w)$ induced by the edge set $E$ as $H_{w;i_1,\ldots,i_{\ell}}$. It is then clear that
\[
    \lim_{n \to \infty}\frac{\# B_{i_1}(w) \cap \cdots \cap B_{i_{\ell}}(w)}{n^{1 + k}} = A_{H_{w; i_1, \ldots, i_{\ell}}}.
\]
This completes the proof.
\end{proof}
\begin{remark}\label{rem:Gw}
    Note that $H_{w;0, \ldots, 2k - 1} = G(w)$.
\end{remark}

\begin{example}
    Consider the case $k = 3$ and the word $w = aabbcc$. $G(w)$ is depicted in Figure~\ref{fig:gwpic}. Let us find $H_{w; 1, 3}$. Note that while creating $G(w)$ we identified vertices $0$, $2$ and $4$, because of the (C2) constraints:
\[
    \pi(0) = \pi(2) = \pi(4).
\]
Now, in $B_1(w)\cap B_3(w)$, we have the gcd constraints
\[
    \gcd(\pi(1), \pi(2)) = \gcd(\pi(3), \pi(4)) = 1.
\]
These place gcd restrictions on the edges $0 \,\, \mathemdash \,\, 1$ and $0 \,\, \mathemdash \,\, 3$ and we obtain as $H_{w; 1, 3}$ the graph shown on Figure~\ref{fig:hwpic}.
\end{example}

\begin{figure}[ht]
    \centering
    \includegraphics[scale = 1]{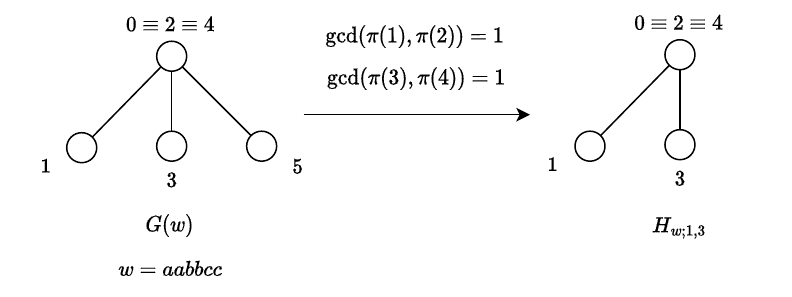}
    \caption{A pictorial description of how $H_{w; 1, 3}$ is obtained from $G(w)$ for $w = aabbcc$.}
    \label{fig:hwpic}
\end{figure}

We can now state our results on the LSDs of the visible and the invisible Wigner matrices.
\begin{theorem}\label{theorem:visiblewigner}
    Let $\widetilde W_n$ be the $ n \times n $ visible Wigner matrix with the entries  $\{x_{i,j,n}: 1 \le i \le j, \ j \geq 1, n \geq 1\} $ which satisfy Assumption~\ref{assmp:entries}. Then the EESD of $\{n^{-1/2}\widetilde W_n\}$ converges weakly to a probability law, say $\mu_{\widetilde{W}}$,  with compact support. Its ESD also converges to the same law weakly almost surely. This law is symmetric about $0$ and its even moments are given by
\begin{equation}\label{eq:moments}
    m_{2k}^{(\mathrm{VW})} = \sum_{w\in \mathcal{C}_{2k}} A_{G(w)} = \sum_{T \in \mathcal{T}_{k + 1}} n(T) A_T,
\end{equation}
where $\mathcal{T}_{k + 1}$ denotes the set of all unlabeled trees on $k + 1$ vertices, and $n(T)$ denotes the number of Catalan words $w$ of length $2k$ such that the unlabeled version of $G(w)$ is $T$.
\end{theorem}

\begin{theorem}\label{theorem:invisiblewigner}
    Let $\widehat{W}_n$ be the  $ n \times n $ invisible Wigner matrix with the entries  $\{x_{i,j,n}: 1 \le i \le j, \ j \geq 1, n \geq 1\} $ which satisfy Assumption~\ref{assmp:entries}. Then the EESD of $\{n^{-1/2}\widehat{W}_n\}$ converges weakly to a probability law, say $\mu_{\widehat{W}}$, with compact support. Its ESD also converges to the same law weakly almost surely. This law is symmetric about $0$ and its even moments are given by
\begin{equation}\label{eq:inmoments}
    m_{2k}^{(\mathrm{IVW})}= \sum_{w \in \C_{2k}} p_{\hat{W}}(w),
\end{equation}
where
\[
    p_{\hat{W}}(w) = 1 + \sum_{\ell = 1}^{2k} (-1)^{\ell}\sum_{0 \le i_1 < \cdots < i_{\ell} \le 2k - 1} A_{H_{w; i_1, \ldots, i_{\ell}}}.
\]
\end{theorem}
\begin{remark}
    (a)  Note that the limits in the two results above do not depend on the underlying probability law of the variables, i.e. the limits are  \textit{universal}.\vskip3pt

    \noindent (b) Suppose the entries are i.i.d. with mean $0$ and variance $1$. Then Assumption~\ref{assmp:entries} does not hold. However, once we prove Theorems \ref{theorem:visiblewigner} and \ref{theorem:invisiblewigner},  by using standard truncation arguments, we can show that in this case also the ESDs of $\{n^{-1/2}\widetilde W_n\}$ and $\{n^{-1/2}\widehat W_n\}$ converge to $\mu_{\widetilde{W}}$ and $\mu_{\widehat{W}}$ weakly almost surely, even if moments higher than 2 are not finite. We omit the details of the proof and refer the reader to \cite{bose2018} for such general truncation arguments.
\end{remark}
\begin{proof}[Proof of Theorem \ref{theorem:visiblewigner}]
    Recall Remark \ref{remark:wignerequiv}. We thus need to consider only even moments. From Facts~\ref{fact:pm_only}, \ref{fact:pw} and \ref{fact:m4}, no words outside $\mathcal{C}_{2k}$ contributes to the limit. Since there are a lot more $0$ entries due to the gcd constraints, unlike the standard Wigner matrix where each Catalan word contributes $1$ (see (\ref{eq:pistarwigner})), now they all contribute strictly less than $1$. Thus condition (U) follows immediately. All we are then left to show is that $p_{\widetilde{W}}(w)$ exists for any $w \in \mathcal{C}(2k)$.

Note that
\[
    \Pi_{\widetilde{W}}^{\prime}(w) = \cap_{i = 0}^{2k - 1} B_i(w).
\]
It follows from Lemma~\ref{lem:limit-subgraph-Gw} and Remark~\ref{rem:Gw} that
\begin{align*}
    p_{\widetilde{W}}(w) &= \lim_{n \to \infty} \frac{1}{n^{1 + k}} \#\Pi_{\widetilde{W}}^{\prime}(w) \\
    &= \lim_{n \to \infty} \frac{\# B_0(w)\cap \cdots \cap B_{2k - 1}(w)}{n^{1 + k}} \\
    &= A_{H_{w;0, \ldots, 2k - 1}} = A_{G(w)}.
\end{align*}
As a consequence,
\[
    \E \,\tr (n^{-1/2}\widetilde{W}_{n})^{2k} \rightarrow \sum_{w \in \C_{2k}}A_{G(w)}.
\]
That $\mu_{\widetilde{W}}$ has compact support follows from the observation that its moments are dominated by the moments of the semi-circle law, which has compact support. This completes the proof.
\end{proof}
\begin{proof}[Proof of Theorem \ref{theorem:invisiblewigner}]
From our discussions earlier, it is enough to show that the limit of the even moments of the EESD exist. Define
\[
    I_{2k} := \{\pi \in [n]^{[2k]} \mid  \gcd(\pi(i), \pi(i + 1)) \ne 1 \text{ for } i = 0, \ldots, 2k - 1\}.
\]
Then $\Pi_{\hat{W}}'(w) = \Pi_{W}'(w) \cap I_{2k}$
We thus have the representation
\[
    m_{2k}^{\mathrm{(IVW)}} = \sum_{w \in \C_{2k}} \lim_{n \to \infty}\frac{1}{n^{1 + k}} \# \Pi_W^{\prime}(w) \cap I_{2k},
\]
assuming the limits inside the sum exist.

By the inclusion-exclusion principle,
\[
    \#\Pi_W^{\prime}(w) \cap I_{2k} = \#\Pi_W^{\prime}(w) + \sum_{\ell = 1}^{2k} (-1)^{\ell} \sum_{0 \le i_1 < \ldots < i_{\ell} \le 2k - 1} \#B_{i_1}(w) \cap \cdots B_{i_{\ell}}(w).
\]
The result now follows from Lemma~\ref{lem:limit-subgraph-Gw}. The claim of compact support follows as before.
\end{proof}
\begin{example}
For $k = 1$, we have $\C_{2k} = \{aa\}$, and
\begin{align*}
    m_{2}^{(\mathrm{IVW})} &= 1 - H_{aa; 0} - H_{aa; 1} + H_{aa; 0, 1} \\
    &= 1 - A_{G(aa)} = 1 - \frac{6}{\pi^2},
\end{align*}
where we have used the fact that
\[
    H_{aa; 0} = H_{aa; 1} = H_{aa; 0, 1} = G(aa)\,.
\]
\end{example}
\begin{remark}\label{rem:joint-conv}
It is well known that independent standard Wigner matrices are asymptotically free in the sense of joint tracial convergence. Joint tracial convergence of independent visible and invisible Wigner matrices can be shown by extending the above arguments. By computing the first few joint moments directly, it can be verified that these matrices are not asymptotically free. We omit the details. In particular, it may be verified that $n^{-1/2}\widehat{W}_n$ and $n^{-1/2}\widetilde{W}_n$ converge jointly in the tracial sense to $(a, b)$, say, where $a$ and $b$ are not free, and $a+b$ is a semi-circle variable.
\end{remark}

In Figure \ref{fig:momcum} we have plotted the first five empirical moments and free cumulants of these three matrices. It is evident that the fourth free cumulant for the Wigner matrix (which is exactly zero) does not equal the sum of the fourth free cumulants for the visible and invisible Wigner matrices (both of which seem to be positive). This empirically demonstrates the non-freeness result stated in Remark~\ref{rem:joint-conv}.
\begin{figure}[ht]
    \centering
    \includegraphics[scale = 0.27]{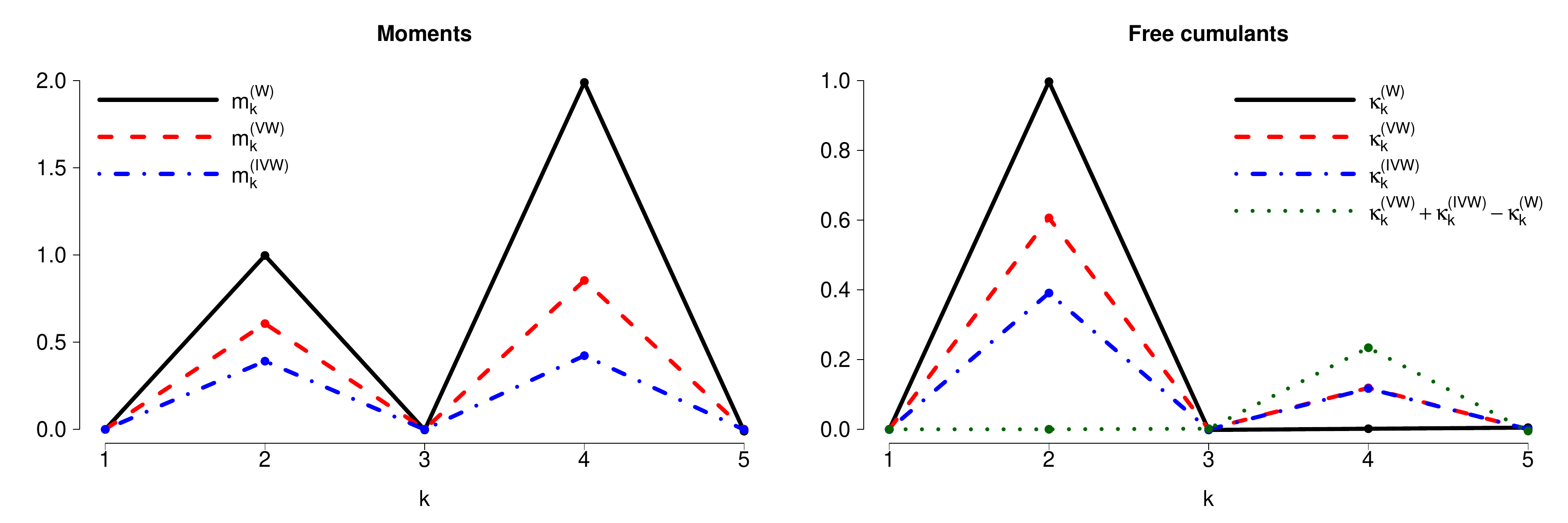}
    \caption{First five empirical moments and \textit{free cumulants} of the ESDs of $n^{-1/2}W_n$, $n^{-1/2}\tilde{W}_n$ and $n^{-1/2}\hat{W}_n$ with standard Gaussian entries, $n = 1000$.}
    \label{fig:momcum}
\end{figure}
\section{Discussion}
It appears difficult to obtain nice formulas for the moments of the two LSDs, let alone obtain the value of $p(w)$ for each $w$. We can compare this situation with the triangular Wigner matrix. Its LSD was shown to exist in \cite{basu2012}. There too there was no success in obtaining the value of $p(w)$ for each $w$. However, by using Abel polynomials, the authors were able to obtain some formula for the moments $\sum_{w} p(w)$, and the LSD was identified.

In the present case, there are many $w \ne w^\prime$ for which the trees $G(w)$ and $G(w^\prime)$ are isomorphic. Identifying the isomorphism classes could result in more compact expressions for the moments of the LSD. Given any Catalan word $w$, we have a double letter in it. Suppose we delete the double letter to get a reduced Catalan word $w^{\prime}$. Any nice relationship between $A_{G(w)}$ and $A_{G(w^{\prime})}$ might help in solving the above question. We currently have only some partial results.

It would be interesting to identify the two LSDs in some manner, or at least derive some qualitative properties. Simulations suggest that they are absolutely continuous with bounded densities (see~Figure~\ref{fig:esd}). Simulations also suggest that (see~Figure~\ref{fig:momcum}), their moments satisfy $m_{2k}^{(\mathrm{IVW})} \le m_{2k}^{(\mathrm{VW})}$, for all $k \ge 1$.

Let $\lambda_{\max}(A)$ denote the maximum eigenvalue of a matrix $A$. It is known that as $n \to \infty$, $\lambda_{\max}(n^{-1/2}W_n)$ converges almost surely to the upper edge of the LSD of $n^{-1/2}W_n$, namely $2$. In a recent work \cite{CHELIOTIS201836} shows that for the asymmetric triangular Wigner matrix $n^{-1/2}W_{n,t}$ say, under certain assumptions, the maximum singular value converges to $e$ almost surely, which is also the upper edge of the LSD of $n^{-1}W_{n,t}W_{n,t}^{\prime}$. From Table \ref{table:maxeval} it appears that the maximum eigenvalues of $n^{-1/2}\widetilde{W}_n$ and $n^{-1/2}\widehat{W}_n$ converge to $\widetilde w \approx 1.7$ and $\widehat{w} \approx 1.5$, respectively. To get a more precise estimate of the actual values of these numbers, more extensive simulations are required. Obtaining the exact value of these limits (which may be irrational numbers) appears to be a difficult problem. We also conjecture that these values are the same as the upper edges of the corresponding LSDs.

\begin{table}[ht]
\centering
\caption{Largest eigenvalues of standard, visible, and invisible Wigner matrices with standard Gaussian entries; $1000$ Monte Carlo simulations with $n = 1000$.}
\begin{tabular}{cccc}
    \toprule
    & $\lambda_{\max}(n^{-1/2}W_n)$ & $\lambda_{\max}(n^{-1/2}\tilde{W}_n)$ & $\lambda_{\max}(n^{-1/2}\hat{W}_n)$ \\
    \midrule
    mean & 1.99 & 1.69 & 1.50 \\
    sd   & 0.01 & 0.01 & 0.01 \\
   \bottomrule
\end{tabular}\label{table:maxeval}
\end{table}

It is known that $n^{2/3}(\lambda_{\max}(n^{-1/2}W_n)-2)$ converges weakly to the GOE Tracy-Widom law $\mathrm{TW}_1$ under appropriate conditions on the entries. Based on the simulations reported in Figure \ref{fig:max_ev}, we conjecture that both $n^{2/3}(\lambda_{\max}(n^{-1/2}\widetilde{W}_n)-\widetilde{w})$ and $n^{2/3}(\lambda_{\max}(n^{-1/2}\widehat{W}_n)-\widehat{w})$ converge in distribution to the $\mathrm{TW}_1$.
\begin{figure}[ht]
    \centering
    \includegraphics[scale = 0.4]{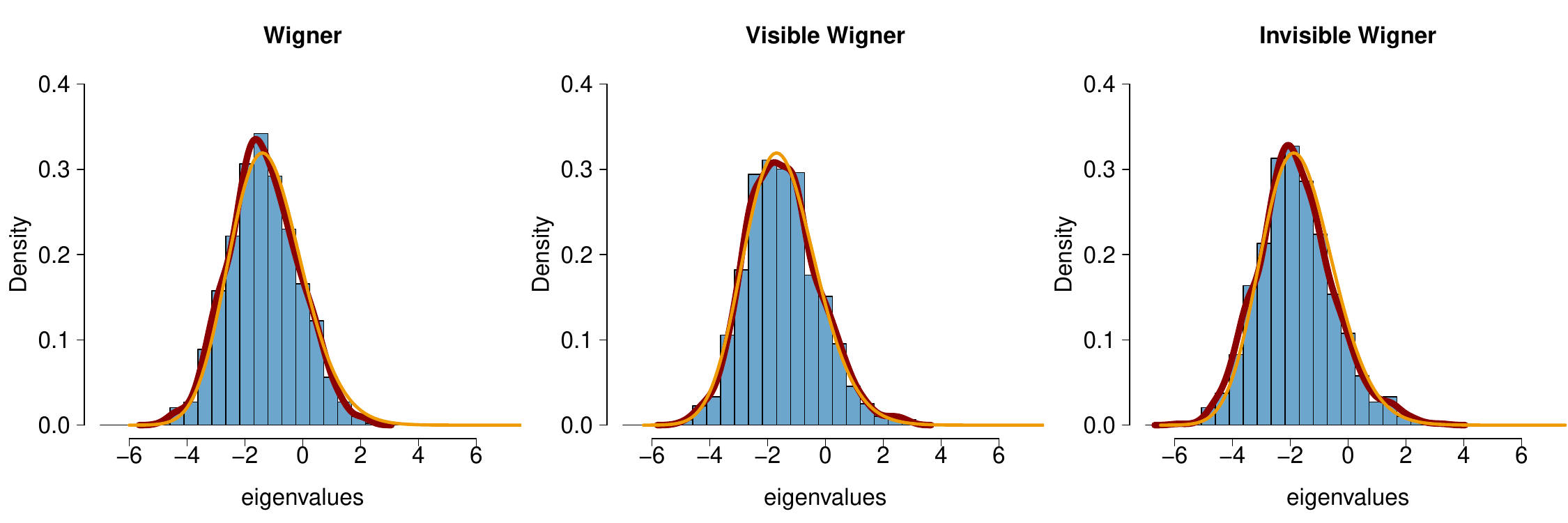}
    \caption{Histogram and kernel density estimates (in red) based on $1000$ Monte Carlo simulations with standard Gaussian entries and $n = 1000$,
    of the ESDs of $n^{2/3}(\lambda_{\max}(n^{-1/2}W_n) - 2)$, $n^{2/3}(\lambda_{\max}(n^{-1/2}\tilde{W}_n) - 1.705)$ and $n^{2/3}(\lambda_{\max}(n^{-1/2}\hat{W}_n) - 1.515)$. The orange curves depict the density of the Tracy-Widom $\mathrm{TW}_1$ distribution.}
    \label{fig:max_ev}
\end{figure}

\section*{Acknowledgements}    
AB was partially supported by his J.C. Bose Fellowship from SERB, Govt.~of India. SSM was partially supported by an INSPIRE research grant (DST/INSPIRE/04/2018/002193) from the Dept.~of Science and Technology, Govt.~of India and a Start-Up Grant from Indian Statistical Institute, Kolkata. 

\bibliographystyle{plain}
\bibliography{bib}

\end{document}